\documentclass{article}
\usepackage[english]{babel}
\usepackage{epsfig}
\usepackage{amsmath}
\usepackage{amsthm}
\usepackage{amssymb}
\usepackage{graphicx}
\usepackage{color}

\newtheorem{theorem}{Theorem}[section]
\newtheorem{definition}[theorem]{Definition}
\newtheorem{lemma}[theorem]{Lemma}
\newtheorem{proposition}[theorem]{Proposition}

\newtheorem{remark}[theorem]{Remark}
\newtheorem{example}[theorem]{Example}

\usepackage[a4paper,left=2.6cm,
right=2.6cm,height=24cm]{geometry}

\newcommand{\hh}{{\mathbb{H}}}
\newcommand{\s}{{\mathbb{S}}}
\newcommand{\cc}{{\mathbb{C}}}
\newcommand{\rr}{{\mathbb{R}}}

\newcommand{\B}{{\mathbb{B}}}
\newcommand{\us}{\mathsf{us}}
\newcommand{\wsh}{\mathsf{wsh}}
\newcommand{\ssh}{\mathsf{ssh}}
\newcommand{\jpsh}{\mathsf{psh}_\jj}
\newcommand{\jj}{{\mathbb{J}}}
\newcommand{\cp}{{\mathbb{CP}}}
\newcommand{\M}{{\mathcal{M}}}

\title{\bf Some notions of subharmonicity over the quaternions}

\author{
Caterina Stoppato\\ 
\small Dipartimento di Matematica e Informatica ``U. Dini'', Universit\`a di Firenze \\
\small Viale Morgagni 67/A, I-50134 Firenze, Italy\\
\small stoppato@math.unifi.it
}

\date{  }
\begin{document}


\maketitle

\begin{abstract}
This works introduces several notions of subharmonicity for real-valued functions of one quaternionic variable. These notions are related to the theory of slice regular quaternionic functions introduced by Gentili and Struppa in 2006. The interesting properties of these new classes of functions are studied and applied to construct the analogs of Green's functions.
\end{abstract}

\vspace{.4em} {\small
\noindent{\bf Acknowledgements.} This work stems from a question posed by Filippo Bracci at Universit\`a di Roma ``Tor Vergata'', where the author was the recipient of the ``Michele Cuozzo'' prize. The author warmly thanks the Cuozzo family, Universit\`a di Roma ``Tor Vergata'' and Filippo Bracci for the remarkable research opportunity.\\
The author is partly supported by GNSAGA of INdAM and by Finanziamento Premiale FOE 2014 ``Splines for accUrate NumeRics: adaptIve models for Simulation Environments'' of MIUR.}

\section{Introduction}

Let $\hh = \rr+i\rr+j\rr+k\rr$ denote the real algebra of quaternions and let 
\[\s := \{q \in \hh : q^2=-1\} = \{\alpha i+\beta j+\gamma k : \alpha^2+\beta^2+\gamma^2=1\}\]
denote the $2$-sphere of quaternionic imaginary units. For each $I \in\s$, the subalgebra $L_I=\rr+I \rr$ generated by $1$ and $I$ is isomorphic to $\cc$. In recent years, this elementary fact has been the basis for the introduction of a theory of quaternionic functions. 

\begin{definition}[\cite{advances}]\label{regularfunction}
Let $f$ be a quaternion-valued function defined on a domain $\Omega$. For each $I \in \s$, let $\Omega_I = \Omega \cap L_I$ and let $f_I = f_{|_{\Omega_I}}$ be the restriction of $f$ to $\Omega_I$. The restriction $f_I$ is called \emph{holomorphic} if it has continuous partial derivatives and
\begin{equation}
\frac{1}{2} \left( \frac{\partial}{\partial x} + I \frac{\partial}{\partial y} \right) f_I(x+yI) \equiv 0.
\end{equation}
The function $f$ is called \emph{(slice) regular} if, for all $I \in \s$, $f_I$ is holomorphic.
\end{definition}

The study of regular quaternionic functions has then grown into a full theory, described in the monograph~\cite{librospringer}. It resembles the theory of holomorphic complex functions, but in a many-sided way that reflects the richness of the non-commutative setting. 

In the present work, we consider several notions of subharmonicity related to the class of regular quaternionic functions. In contrast with the work~\cite{pharmonicity}, which studies the relation between quaternion-valued (or Clifford-valued) regular functions and real harmonicity, we consider real-valued functions of a quaternionic variable and look for new notions of subharmonicity compatible with composition with regular functions.

The first attempt is \emph{$\jj$-plurisubharmonicity}. However, this property is quite restrictive, besides being preserved by composition with a regular function $f$ only if $f$ is \emph{slice preserving}, that is, if $f(\Omega_I)\subseteq L_I$ for all $I\in\s$.

For this reason, the alternative notions of \emph{weakly subharmonic} and \emph{strongly subharmonic} function are introduced. Composition with regular functions turns out to map strongly subharmonic functions into weakly subharmonic ones. Moreover, composition with slice preserving regular functions is proven to preserve weak subharmonicity.

These new notions of subharmonicity turn out to have many nice properties that recall the complex and pluricomplex cases, including mean-value properties and versions of the maximum modulus principle.

These results are finally applied to construct quaternionic analogs of Green's functions, which reveal many peculiarities due to the non-commutative setting.


\section{Prerequisites}\label{sectiondifferential}

Let us recall a few properties of the algebra of quaternions $\hh$, on which we consider the standard Euclidean metric and topology.

\begin{itemize}
\item For each $I \in\s$, the couple $1,I$ can be completed to a (positively oriented) orthonormal basis $1,I,J,K$ by choosing $J \in \s$ with $I \perp J$ and setting $K = IJ$. 
\item The coordinates of any $q\in \hh$ with respect to such a basis can be recovered as
\begin{align*}
x_0(q)&=\frac{1}{4}(q-IqI-J qJ-KqK)\\
x_1(q)&=\frac{1}{4I}(q-I qI+JqJ+KqK)\\
x_2(q)&=\frac{1}{4J}(q+I qI-J qJ+KqK)\\ 
x_3(q)&=\frac{1}{4K}(q+IqI+JqJ-KqK).
\end{align*}
\item Mapping each $v \in T_{q_0}\hh \cong \hh$ to $Iv$ for all $q_0 \in \hh$ defines an (orthogonal) complex structure on $\hh$, called \emph{constant}. A biholomorphism between $(\hh,I)$ and $(L_I^2,I)\cong(\cc^2,i)$ can be constructed by mapping each $q$ to $(z_1(q),z_2(q))$, where
\begin{align*}
z_1(q)&=x_0(q)+Ix_1(q) = (q-IqI) \frac{1}{2},\\
z_2(q)&=x_2(q)+Ix_3(q) =  (q+IqI) \frac{1}{2J},
\end{align*}
are such that $z_1(q)+z_2(q)J = q$. Both $z_1$ and $z_2$ depend on the choice of $I$; $z_2$ also depends on $J$, but only up to a multiplicative constant $c\in L_I$.
\end{itemize}

For every domain $\Omega$ and every function $f : \Omega \to \hh$, let us denote by $f=f_1+f_2 J$ the corresponding decomposition with $f_1,f_2$ ranging in $L_I$. Furthermore $\partial_1,\partial_2, \bar \partial_1, \bar \partial_2:C^1(\Omega,L_I)\to C^0(\Omega,L_I)$ will denote the corresponding complex derivatives. In other words,
\begin{align*}
\partial_1 &=\frac{1}{2}\left(\frac{\partial}{\partial x_0} - I \frac{\partial}{\partial x_1}\right)\\
\bar \partial_1 &=\frac{1}{2}\left(\frac{\partial}{\partial x_0} + I \frac{\partial}{\partial x_1}\right)\\
\partial_2 &=\frac{1}{2}\left(\frac{\partial}{\partial x_2} - I \frac{\partial}{\partial x_3}\right)\\
\bar \partial_2 &=\frac{1}{2}\left(\frac{\partial}{\partial x_2} + I \frac{\partial}{\partial x_3}\right).
\end{align*}
We notice that these derivatives commute with each other, and that $\partial_1,\bar \partial_1$ depend only on $I$, while $\partial_2, \bar \partial_2$ depend on both $I$ and $J$. 

The definition of regular function (Definition~\ref{regularfunction}) amounts to requiring that the restriction to $\Omega_I$ be holomorphic from $(\Omega_I, I)$ to $(\hh,I)$ for all $I \in \s$. Curiously, if the domain is carefully chosen then a stronger property holds.

\begin{definition}
Let $\Omega$ be a domain in $\hh$. $\Omega$ is a \emph{slice domain} if it intersects the real axis $\rr$ and if, for all $I \in \s$, the intersection $\Omega_I$ with the complex plane $L_I$ is a connected. Moreover, $\Omega$ is termed \emph{symmetric} if it is axially symmetric with respect to the real axis $\rr$.
\end{definition}

If we denote by $\partial_cf$ the \emph{slice derivative} 
\begin{equation*}
\partial_cf(x+Iy) = \frac{1}{2} \left( \frac{\partial}{\partial x} - I \frac{\partial}{\partial y} \right) f_I(x+yI) 
\end{equation*}
introduced in \cite{advances} and by $\partial_sf$ the \emph{spherical derivative} 
\begin{equation*}
\partial_sf(q) =(q - \bar q)^{-1} \left(f(q) -f(\bar q)\right)
\end{equation*}
introduced in \cite{perotti}, then the aforementioned property can be stated as follows.

\begin{theorem}[\cite{expansion}]
Let $\Omega$ be a symmetric slice domain, let $f : \Omega \to \hh$ be a regular function and let $q_0 \in \Omega$. Chosen $I,J \in \s$ so that $q_0 \in L_I$ and $I \perp J$, let $z_1,z_2,\bar z_1, \bar z_2$ be the induced coordinates and let $\partial_1,\partial_2,\bar \partial_1, \bar \partial_2$ be the corresponding derivations. Then
\begin{equation}\label{complexholomorphy}
\left.\left( \begin{array}{cc}
\bar\partial_1 f_1 & \bar \partial_2 f_1\\
\bar \partial_1 f_2 & \bar \partial_2 f_2
\end{array}\right)\right|_{q_0} =
\left( \begin{array}{cc}
0 & 0 \\
0 & 0
\end{array}\right).
\end{equation}
Furthermore, if $q_0 \not \in \rr$ then
\begin{equation}\label{complexjacobian}
\left.\left( \begin{array}{cc}
\partial_1 f_1 & \partial_2 f_1 \\
\partial_1 f_2 & \partial_2 f_2
\end{array}\right)\right|_{q_0} =
\left( \begin{array}{cr}
\partial_cf_1(q_0) & - \overline{\partial_sf_2(q_0)} \\
\partial_cf_2(q_0) & \overline{\partial_sf_1(q_0)}
\end{array}\right).
\end{equation}
If, on the contrary, $q_0 \in \rr$ then
\begin{equation}
\left.\left( \begin{array}{cc}
\partial_1 f_1 & \partial_2 f_1 \\
\partial_1 f_2 & \partial_2 f_2
\end{array}\right)\right|_{q_0} =
\left( \begin{array}{cr}
\partial_cf_1(q_0) & - \overline{\partial_cf_2(q_0)} \\
\partial_cf_2(q_0) & \overline{\partial_cf_1(q_0)}
\end{array}\right).
\end{equation}
\end{theorem}

We point out that we have not proven that $f$ is holomorphic with respect to the constant structure $I$: with respect to the basis $1,I,J, IJ$, equality \eqref{complexholomorphy} only holds at those points $q_0$ that lie in $L_I$. In fact, regularity is related to a different notion of holomorphy, which involves non-constant orthogonal complex structures. Let us recall the notations $Re(q) = x_0(q), Im(q) = q-Re(q)$ for $q \in \hh$ and let us set
\begin{equation*}
\jj_{q_0}v := \frac{Im(q_0)}{|Im(q_0)|} v\qquad \forall\ v \in T_{q_0}(\hh\setminus\rr) \cong \hh.
\end{equation*}
Then $\pm\jj$ are orthogonal complex structures on 
\begin{equation*}
\hh\setminus\rr = \bigcup_{I \in \s} (\rr+I\rr^+) 
\end{equation*}
and they are induced by the natural identification with the complex manifold $\cp^1 \times (\rr+i\rr^+)$.

\begin{theorem}[\cite{ocs}]\label{ocs}
Let $\Omega$ be a symmetric slice domain, and let $f: \Omega \to \hh$ be an injective regular function. Then the real differential of $f$ is invertible at each $q \in \Omega$ and the push-forward of $\jj$ via $f$, that is,
\begin{equation}
\jj^f_{f(q)}v =  \frac{Im(q)}{|Im(q)|} v \qquad \forall\ v \in T_{f(q)} f(\Omega\setminus \rr) \cong \hh\,,
\end{equation}
is an orthogonal complex structure on $f(\Omega \setminus \rr)$.
\end{theorem}

In the hypotheses of the previous theorem, $f$ is (obviously) a holomorphic map from $(\Omega \setminus \rr, \jj)$ to $\left(f(\Omega \setminus \rr), \jj^f\right)$. Furthermore, there is a special class of regular functions such that $\jj^f = \jj$.

\begin{remark}\label{J-holomorphic}
Let $f: \Omega \to \hh$ be a \emph{slice preserving} regular function, namely a regular function such that $f(\Omega_I)\subseteq L_I$ for all $I \in \s$. Then $f(\Omega \setminus \rr) = f(\Omega) \setminus \rr$ and $f$ is a holomorphic map from $(\Omega \setminus \rr, \jj)$ to $\left( f(\Omega) \setminus \rr, \jj\right)$
\end{remark}


\section{Quaternionic notions of subharmonicity}\label{sectionsubharmonicity}

Let $\Omega$ be a domain in $\hh$ and let 
\begin{equation*}
\us(\Omega) = \{u: \Omega \to [-\infty, +\infty), u \mathrm{\ upper\ semicontinuous,\ } u \not \equiv -\infty\}.
\end{equation*}
For $u \in \us(\Omega)$, we aim at defining some notion of subharmonicity that behaves well when we compose $u$ with a regular function. Remark \ref{J-holomorphic} encourages us to consider \emph{$\jj$-plurisubharmonic} and \emph{$\jj$-harmonic} functions, i.e., functions that are pluri(sub)harmonic with respect to the complex structure $\jj$. However, the notion of $\jj$-plurisubharmonicity on a symmetric slice domain $\Omega$ is induced by plurisubharmonicity in $\cp^1 \times D_{\Omega}$ with \begin{equation*}
D_{\Omega}=\{x+iy \in \rr+i\rr^+ : x+y\s \subset\Omega\},
\end{equation*}
which amounts to constance in the first variable and subharmonicity in the second variable. We conclude:

\begin{proposition}
Let $\Omega$ be a symmetric domain in $\hh$. A function $u \in \us(\Omega)$ is $\jj$-plurisubharmonic in $\Omega \setminus \rr$ if, and only if, there exists a subharmonic function $\upsilon : D_{\Omega} \to \rr$ such that $u(x+Iy) = \upsilon(x+iy)$ for all $I \in \s$ and for all $x+iy 	\in  D_{\Omega}$. In particular, a function $u \in C^2(\Omega,\rr)$ is $\jj$-plurisubharmonic in $\Omega \setminus \rr$ if, and only if, $u(x+Iy)$ does not depend on $I$ and
\begin{equation}\label{weaksubharmonicity}
\left(\frac{\partial^2}{\partial x^2} + \frac{\partial^2}{\partial y^2} \right) u(x+Iy) \geq 0.
\end{equation}
\end{proposition}

Analogous observations can be made for $\jj$-harmonicity. In order to get a richer class of functions, we need to suitably weaken the notion of subharmonicity considered. We are thus encouraged to give the following definition.

\begin{definition}
Let $\Omega$ be a domain in $\hh$ and let $u \in \us(\Omega)$. We call $u$ \emph{weakly subharmonic} if for all $I \in \s$ the restriction $u_I=u_{|_{\Omega_I}}$  is subharmonic (after the natural identification between $L_I$ and $\cc$). We say that $u$ is \emph{weakly harmonic} if, for all $I \in \s$, $u_I$  is harmonic.
\end{definition}

\begin{remark}
A function $u \in C^2(\Omega,\rr)$ is weakly subharmonic if, and only if, for all $I \in \s$, the associated $\partial_1,\bar\partial_1$ are such that $\bar \partial_1 \partial_1u_I\geq 0$ in $\Omega_I$; that is, inequality \eqref{weaksubharmonicity} holds for all $I \in \s$. Furthermore, $u$ is weakly harmonic if, and only if, equality holds at all points.
\end{remark}

By construction:

\begin{proposition}
Let $\Omega$ be a symmetric domain in $\hh$ and let $u \in \us(\Omega)$. If $u$ is $\jj$-pluri(sub)harmonic in $\Omega \setminus \rr$ then it is weakly (sub)harmonic in $\Omega \setminus \rr$. If, moreover, $u$ is continuous at all points of $\Omega\cap\rr$ then $u$ is weakly (sub)harmonic in $\Omega$.
\end{proposition}

The converse implication is not true, as shown by the next example. Here, $\langle\cdot,\cdot\rangle$ denotes the Euclidean scalar product on $Im(\hh)\cong\rr^3$.

\begin{example}\label{coordinates}
All real affine functions $u: \hh \to \rr$ are weakly harmonic, including the coordinates $x_0,x_1,x_2,x_3$ with respect to any basis $1,I,J,IJ$ with $I,J \in \s, I\perp J$. On the other hand, $x_1,x_2,x_3$ are not $\jj$-plurisubharmonic in $\hh \setminus \rr$, as $x_1(x+Iy) = y \langle I,i \rangle, x_2(x+Iy) = y \langle I,j \rangle, x_3(x+Iy) = y \langle I,k \rangle$ are not constant in $I$.
\end{example}

Actually, a stronger property holds for real affine functions $u: \hh \to \rr$: they are pluriharmonic with respect to any constant orthogonal complex structure. This motivates the next definition.

\begin{definition}
Let $\Omega$ be a domain in $\hh$ and let $u \in \us(\Omega)$. We say that $u$ is \emph{strongly (sub)harmonic} if it is pluri(sub)harmonic with respect to every constant orthogonal complex structure on $\Omega$.
\end{definition}

\begin{remark}
A function $u \in C^2(\Omega,\rr)$ is strongly subharmonic if, for all $I,J \in \s$ with $I \perp J$,
\begin{equation}
H_{I,J}(u)=
\left( \begin{array}{cc}
\bar\partial_1\partial_1 u & \bar\partial_1\partial_2 u \\
\bar\partial_2\partial_1 u & \bar\partial_2\partial_2 u
\end{array}\right)
\end{equation}
is a positive semidefinite matrix at each $q \in \Omega$. The function $u$ is strongly harmonic if for all $I,J \in \s$ with $I \perp J$ the matrix $H_{I,J}(u)$ has constant rank $0$.
\end{remark}

Clearly, if the matrix $H_{I,J}(u)$ is positive semidefinite then its $(1,1)$-entry $\bar \partial_1 \partial_1u$ is non-negative. Similarly, if $H_{I,J}(u)$ has constant rank $0$ then $\bar \partial_1 \partial_1u\equiv 0$. This leads to the next result, which, however, is not only true for $u \in C^2(\Omega,\rr)$ but also for $u \in \us(\Omega)$.

\begin{proposition}
Let $u \in \us(\Omega)$. If $u$ is strongly subharmonic then it is weakly harmonic.
\end{proposition}

\begin{proof}
Let $u \in \us(\Omega)$ be strongly subharmonic, let $I\in\s$ and let us prove that $u_I$ is subharmonic. By construction, $u$ is plurisubharmonic with respect to the constant orthogonal complex structure $I$. Moreover, the inclusion map $incl:\Omega_I\to\Omega$ is a holomorphic map from $(\Omega_I,I)$ to $(\Omega,I)$. As a consequence, $u_I=u\circ incl$ is subharmonic, as desired.
\end{proof}

Example \ref{coordinates} shows that strong (sub)harmonicity does not imply $\jj$-pluri(sub)harmonicity. The converse implication does not hold, either:
\begin{example}\label{x^2-y^2}
The function $u(q) = Re(q^2)$ is $\jj$-pluriharmonic in $\hh\setminus\rr$, as $u(x+Iy) = x^2-y^2$ does not depend on $I$ and $\left(\frac{\partial^2}{\partial x^2} + \frac{\partial^2}{\partial y^2} \right) u(x+Iy) \equiv 0$. On the other hand, $u$ is not strongly subharmonic. Actually, it is not plurisubharmonic with respect to any constant orthogonal complex structure $I$: after choosing $J \in \s$ with $J \perp I$, we compute
\begin{equation*}
u(z_1+z_2J) = Re\left(z_1^2-z_2\bar z_2 +(z_1z_2+z_2\bar z_1)J\right) = \frac{z_1^2+\bar z_1^2}{2}-z_2\bar z_2
\end{equation*}
for all $z_1,z_2 \in L_I$, so that $H_{I,J}(u)\equiv
\left( \begin{array}{cc}
0 & 0\\
0 & -1
\end{array}\right)$. 
\end{example}

We have proven the following implications (none of which can be reversed):
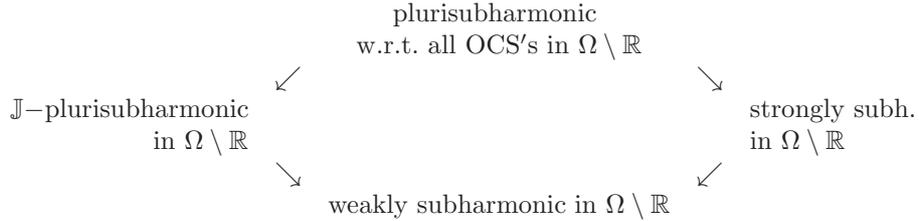
\begin{equation*}
\begin{array}{rlcrl}
&&\mathrm{plurisubharmonic\ }&&\\
&&\mathrm{w.r.t.\ all\ OCS's\ in\ }\Omega \setminus \rr&&\\
&\swarrow&&\searrow&\\
\jj\mathrm{-plurisubharmonic}&&&&\mathrm{strongly\ subh.}\\
\mathrm{in\ } \Omega \setminus \rr&&&&\mathrm{in\ } \Omega\setminus \rr\\
&\searrow&&\swarrow&\\
&&\mathrm{weakly\ subharmonic\ in\ } \Omega\setminus \rr&&
\end{array}
\end{equation*}

A similar scheme can be drawn for the quaternionic notions of harmonicity. We show with a further example that a strongly subharmonic function is not necessarily strongly \emph{harmonic} when it is weakly harmonic, or even $\jj$-pluriharmonic.

\begin{example}\label{log}
Consider the function $u : \hh \to \rr$ with $u(q):=\log|q|$ for all $q\in\hh\setminus\{0\}$ and $u(0):=-\infty$. $u$ is $\jj$-pluriharmonic in $\hh\setminus\rr$, as $u(x+Iy) = \frac12 \log(x^2+y^2)$. As a consequence, $u$ is also weakly harmonic. 
On the other hand, for any choice of $I,J \in \s$, the fact that $u(z_1+z_2J) = \frac{1}{2}\log(z_1\bar z_1 + z_2 \bar z_2)$ implies that 
\begin{align*}
H_{I,J}(u)_{|_{z_1+z_2J}} &= \frac{1}{2 (z_1\bar z_1 + z_2 \bar z_2)^2} 
\left( \begin{array}{cc}
z_2 \bar z_2 & -z_1 \bar z_2\\
-z_2\bar z_1 & z_1\bar z_1
\end{array}\right)\\
&= \frac{1}{2 (|z_1|^2 + |z_2|^2)^2} 
\left( \begin{array}{cc}
|z_2|^2 & -z_1 \bar z_2\\
-z_2\bar z_1 & |z_1|^2
\end{array}\right)\,.
\end{align*}
Hence, $u$ is strongly subharmonic but it is not strongly harmonic.
\end{example}

Let us review a few classical constructions in our new environment.

\begin{remark}
On a given domain $\Omega$, let us denote by $\wsh(\Omega)$ the set of weakly subharmonic functions, by $\ssh(\Omega)$ that of strongly subharmonic functions, and by $\jpsh(\Omega)$ that of $\jj$-plurisubharmonic functions on $\Omega$ (if $\Omega$ equals a symmetric domain minus $\rr$). If $S$ is any of these sets then:
\begin{enumerate}
\item $S$ is a convex cone;
\item for all $u \in S$, if $\varphi$ is a real-valued $C^2$ function on a neighborhood of $u(\rr)$ and if $\varphi$ is increasing and convex then $\varphi \circ u :\Omega \to \rr$ also belongs to $S$;
\item for all $u_1,u_2 \in S$, the function $u(q) = \max\{u_1(q),u_2(q)\}$ belongs to $S$;
\item if $\{u_\alpha\}_{\alpha \in A}$ (with $A \neq \emptyset$) is a family in $S$, locally bounded from above, and if $u(q) = \sup_{\alpha \in A} u_\alpha(q)$ for all $q \in \Omega$ then the upper semicontinuous regularization $u^*$ belongs to $S$.
\end{enumerate}
\end{remark}

\begin{example}
For any $\alpha>0$, the function $u : \hh \to \rr \ q \mapsto |q|^\alpha$ is strongly subharmonic in $\hh$ and it is $\jj$-plurisubharmonic in $\hh \setminus \rr$.
\end{example}

\begin{example}
The functions $Re^2(q) = x_0^2(q)$ and $|Im(q)|^2 = x_1^2(q)+x_2^2(q)+x_3^2(q)$ are strongly subharmonic in $\hh$. (They are also $\jj$-plurisubharmonic in $\hh\setminus\rr$, as $Re^2(x+Iy) = x^2$ and $|Im(x+Iy)|^2=y^2$).
\end{example}

We conclude this section showing that any given subharmonic function on an axially symmetric planar domain extends to a weakly subharmonic function on the corresponding symmetric domain of $\hh$.

\begin{remark}
If we start with a domain $D \subseteq \cc$ that is symmetric with respect to the real axis and a (sub)harmonic function $\upsilon$ on $D$, we may define a weakly (sub)harmonic function $u$ on the symmetric domain $\Omega = \bigcup_{x+iy \in D} x+y\s$ by setting
\begin{equation*}
u(x+Iy) := \frac{1+\langle I,i \rangle}2 \upsilon(x+iy) + \frac{1-\langle I,i \rangle}2 \upsilon(x-iy)
\end{equation*} 
for all $x \in \rr,I\in \s, y>0$ such that $x+Iy \in \Omega$ and $u(x):=\upsilon(x)$ for all $x\in\Omega\cap\rr$.
\end{remark}


\section{Composition with regular functions}\label{sectioncomposition}

We now want to understand the behavior of the different notions of subharmonicity we introduced, under composition with regular functions. For $\jj$-pluri(sub)harmonicity, Remark \ref{J-holomorphic} immediately implies:

\begin{proposition}
Let $\Omega$ be a symmetric domain in $\hh$ and let $u \in \us(\Omega)$. $u$ is $\jj$-pluri(sub)harmonic in $\Omega\setminus\rr$ if, and only if, for every slice preserving regular function $f : \Omega'\setminus\rr \to \Omega\setminus\rr$, the composition $u \circ f$ is $\jj$-pluri(sub)harmonic in $\Omega'\setminus\rr$.
\end{proposition}

It is essential to restrict to slice preserving regular functions. If we compose $u$ with an injective regular function $f$ then the only sufficient condition we know in order for $u\circ f$ to be $\jj$-pluri(sub)harmonic is, that $u$ be $\jj^f$-pluri(sub)harmonic (see Theorem \ref{ocs}).

For weak (sub)harmonicity, we can prove the next result.

\begin{theorem}\label{compositionweakregular}
Let $u \in \us(\Omega)$. $u$ is weakly (sub)harmonic in $\Omega$ if, and only if, for every slice preserving regular function $f : \Omega' \to \Omega$, the composition $u \circ f$ is weakly (sub)harmonic in $\Omega'$.
\end{theorem}

\begin{proof}
If $u \circ f$ is weakly (sub)harmonic for all slice preserving regular $f$ then in particular $f=f\circ id$ is weakly (sub)harmonic.

Conversely, let $u \in \us(\Omega)$ be weakly (sub)harmonic, let $f : \Omega' \to \Omega$ be a slice preserving regular function and let us prove that $u \circ f$ is weakly (sub)harmonic. For each $I\in\s$, $u_I$ is (sub)harmonic in $\Omega_I$ and the restriction $f_I$ is a holomorphic map from $(\Omega'_I,I)$ to $(\Omega_I,I)$. As a consequence, $(u \circ f)_I=u_I\circ f_I$ is (sub)harmonic in $\Omega'_I$.
\end{proof}

As for strong harmonicity and subharmonicity, they are preserved under composition with quaternionic affine transformations, as the latter are holomorphic with respect to any constant structure $I \in \s$:

\begin{remark}\label{affine}
If $u$ is a strongly (sub)harmonic function on a domain $\Omega \subseteq \hh$ then, for any $a,b \in \hh$ with $b\neq0$, the function $v(q)=u(a+qb)$ is strongly (sub)harmonic in $\Omega b^{-1}-a$.
\end{remark}

However, strong harmonicity and subharmonicity are not preserved by composition with other regular functions, not even slice preserving regular functions (see Example \ref{x^2-y^2}). For this reason, we address the study of their composition with regular functions by direct computation, starting with the $C^2$ case.

\begin{lemma}\label{laplacian}
Let $I,J \in\s$  with $I \perp J$ and let us consider the associated $\partial_1,\partial_2, \bar \partial_1, \bar \partial_2$. Let $\Omega$ be a domain in $\hh$ and let $u \in C^2(\Omega,\rr)$. For every symmetric slice domain $\Omega'$ and for every regular function $f : \Omega' \to \Omega$, we have
\begin{equation}
\bar \partial_1 \partial_1(u\circ f)_{|_{q_0}} =
(\overline{\partial_1 f_1}, \overline{\partial_1f_2})|_{q_0}\cdot
H_{I,J}(u)_{|_{f(q_0)}}
\cdot
\left.\left( \begin{array}{c}
\partial_1 f_1\\
\partial_1 f_2
\end{array}\right)\right|_{q_0}
\end{equation}
at each $q_0 \in \Omega'_I$.
If, moreover, $f$ is slice preserving then
\begin{equation}
\bar \partial_1 \partial_1(u\circ f)_{|_{q_0}} = |\partial_1 f_1|^2_{|_{q_0}} \cdot \bar\partial_1\partial_1 u_{|_{f(q_0)}}
\end{equation}
at each $q_0 \in \Omega'_I$. The same is true if there exists a constant $c \in \hh$ such that $f_I+c$ maps $\Omega'_I$ to $L_I$.
\end{lemma}

\begin{proof}
We compute:
\[\partial_1(u\circ f) = (\partial_1 u) \circ f \cdot \partial_1f_1 + (\partial_2 u) \circ f \cdot \partial_1f_2 + (\bar \partial_1 u) \circ f \cdot \partial_1\bar f_1 + (\bar \partial_2 u) \circ f \cdot \partial_1\bar f_2\]
and
\begin{align*}
\bar \partial_1 \partial_1(u\circ f) =\ & \bar \partial_1((\partial_1 u) \circ f) \cdot \partial_1f_1 + (\partial_1 u) \circ f \cdot \bar \partial_1\partial_1f_1+\\
&+  \bar \partial_1((\partial_2 u) \circ f) \cdot \partial_1f_2 + (\partial_2 u) \circ f \cdot  \bar \partial_1\partial_1f_2 +\\
&+ \bar \partial_1((\bar \partial_1 u) \circ f) \cdot \partial_1\bar f_1 +(\bar \partial_1 u) \circ f \cdot \bar \partial_1\partial_1\bar f_1 +\\
&+ \bar \partial_1((\bar \partial_2 u) \circ f) \cdot \partial_1\bar f_2 + (\bar \partial_2 u) \circ f \cdot \bar \partial_1\partial_1\bar f_2\,.
\end{align*}
If we evaluate the previous expression at a point $q \in L_I$, equality \eqref{complexholomorphy} guarantees the vanishing of all terms but the first and the third. Hence,
\[\bar \partial_1 \partial_1(u\circ f)_{|_q} = \bar \partial_1((\partial_1 u) \circ f)_{|_q} \cdot {\partial_1f_1}_{|_q}+  \bar \partial_1((\partial_2 u) \circ f)_{|_q} \cdot {\partial_1f_2}_{|_q}\]
where
\begin{align*}
&\bar \partial_1((\partial_1 u) \circ f)_{|_q} =\\
&=\partial_1\partial_1 u _{|_{f(q)}} \cdot {\bar \partial_1 f_1}_{|_q}+ \partial_2\partial_1 u _{|_{f(q)}} \cdot {\bar \partial_1 f_2}_{|_q}+\bar\partial_1\partial_1 u _{|_{f(q)}} \cdot {\bar \partial_1 \bar f_1}_{|_q} + \bar\partial_2\partial_1 u _{|_{f(q)}} \cdot {\bar \partial_1 \bar f_2}_{|_q} =\\
&= \bar\partial_1\partial_1 u _{|_{f(q)}} \cdot \overline{\partial_1 f_1}_{|_q} + \bar\partial_2\partial_1 u _{|_{f(q)}} \cdot \overline{\partial_1 f_2}_{|_q}
\end{align*}
and
\begin{align*}
&\bar \partial_1((\partial_2 u) \circ f)_{|_q} =\\
&=\partial_1\partial_2 u _{|_{f(q)}} \cdot {\bar \partial_1 f_1}_{|_q}+ \partial_2\partial_2 u _{|_{f(q)}} \cdot {\bar \partial_1 f_2}_{|_q}+\bar\partial_1\partial_2 u _{|_{f(q)}} \cdot {\bar \partial_1 \bar f_1}_{|_q} + \bar\partial_2\partial_2 u _{|_{f(q)}} \cdot {\bar \partial_1 \bar f_2}_{|_q} =\\
&=\bar\partial_1\partial_2 u _{|_{f(q)}} \cdot \overline{\partial_1 f_1}_{|_q} + \bar\partial_2\partial_2 u _{|_{f(q)}} \cdot \overline{\partial_1 f_2}_{|_q}\,.
\end{align*}
Thus,
\begin{align*}
\bar \partial_1 \partial_1(u\circ f)_{|_q} =\ &\bar\partial_1\partial_1 u _{|_{f(q)}} \cdot |\partial_1 f_1|^2_{|_q} + \bar\partial_2\partial_1 u _{|_{f(q)}} \cdot \overline{\partial_1 f_2}_{|_q}\cdot {\partial_1f_1}_{|_q} +\\
&+\bar\partial_1\partial_2 u _{|_{f(q)}} \cdot \overline{\partial_1 f_1}_{|_q}{\partial_1f_2}_{|_q} + \bar\partial_2\partial_2 u _{|_{f(q)}} \cdot |\partial_1 f_2|^2_{|_q}\,;
\end{align*}
that is,
\[\bar \partial_1 \partial_1(u\circ f)_{|_q}= (\overline{\partial_1 f_1}, \overline{\partial_1f_2})|_q\cdot
\left. \left( \begin{array}{cc}
\bar\partial_1\partial_1 u & \bar\partial_1\partial_2 u \\
\bar\partial_2\partial_1 u & \bar\partial_2\partial_2 u
\end{array}\right)\right|_{f(q)} \cdot
\left.\left( \begin{array}{c}
\partial_1 f_1\\
\partial_1 f_2
\end{array}\right)\right|_{q}\,.\]
Finally, if there exists $c = c_1+c_2J \in \hh$ such that $f_I+c$ maps $\Omega_I$ to $L_I$. then $f_2\equiv -c_2$ in $\Omega_I$ so that $\partial_1 f_2$ vanishes identically in $\Omega_I$ and
\[\bar \partial_1 \partial_1(u\circ f)_{|_{q}} = |\partial_1 f_1|^2_{|_{q}} \cdot \bar\partial_1\partial_1 u_{|_{f(q)}}\,,\]
as desired.
\end{proof}

We are now ready to study the composition of strongly (sub)harmonic $C^2$ functions with regular functions.

\begin{theorem}\label{compositionstrongregularc2}
Let $u\in C^2(\Omega,\rr)$. $u$ is strongly (sub)harmonic if, and only if, for every symmetric slice domain $\Omega'$ and for every regular function $f : \Omega' \to \Omega$, the composition $u \circ f$ is weakly (sub)harmonic.
\end{theorem}

\begin{proof}
If $u : \Omega \to \rr$ is strongly subharmonic then, for all $I,J \in \s$ (with $I \perp J$), the matrix $H_{I,J}(u)$ is positive semidefinite. For every regular function $f : \Omega' \to \Omega$ and for all $I \in \s$, Lemma \ref{laplacian} implies $\bar \partial_1 \partial_1(u\circ f)_{|_{z_1}}\geq 0$ at each $z_1 \in \Omega_I$. Hence, $u\circ f$ is weakly subharmonic.

Conversely, if $u \circ f$ is weakly subharmonic for every regular function $f : \Omega' \to \Omega$ then we can prove that $u$ is strongly subharmonic in the following way. Let us fix $I,J \in \s, p \in \Omega$ (with $I \perp J$) and prove that $H_{I,J}(u)$ is positive semidefinite at $p$, i.e., that 
\[(\bar v_1, \bar v_2)\cdot
H_{I,J}(u)_{|_{p}}
\cdot
\left( \begin{array}{c}
v_1\\
v_2
\end{array}\right)
\geq 0\]
for arbitrary $v_1,v_2 \in L_I$. Let us set $v := v_1+v_2J$ and $f(q):=qv+p$ for $q\in B(0,R)$ (with $R>0$ small enough to guarantee the inclusion of $f(B(0,R))=B(p,|v|R)$ into $\Omega$). By direct computation, $\partial_cf\equiv v$. Formula \eqref{complexjacobian} yields the equalities ${\partial_1 f_1}_{|_{0}}= v_1, {\partial_1 f_2}_{|_{0}}= v_2$. Taking into account Lemma \ref{laplacian} and the fact that $f(0)=p$, we conclude that
\[(\bar v_1, \bar v_2)\cdot
H_{I,J}(u)_{|_{p}}
\cdot
\left( \begin{array}{c}
v_1\\
v_2
\end{array}\right)
= \bar \partial_1 \partial_1(u\circ f)_{|_{0}}\,.\]
Since $u\circ f$ is weakly subharmonic, $\bar \partial_1 \partial_1(u\circ f)_{|_{0}}\geq0$ and we have proven the desired inequality.

Analogous reasonings characterize strong \emph{harmonicity}.
\end{proof}

The previous result allows us to construct a large class of examples of weakly subharmonic functions.

\begin{example}
For any regular function $f: \Omega \to \hh$ on a symmetric slice domain $\Omega$, the components of $f$ with respect to any basis $1,I,J, IJ$ with $I,J \in \s,I\perp J$ are weakly harmonic. Furthermore, for all $\alpha>0$ the functions $\log|f|, |f|^\alpha,Re^2f, |Im f|^2$ are weakly subharmonic.
\end{example}


\section{Mean-value property and consequences}\label{sectionmeanvalue}

We can characterize weak and strong pluri(sub)harmonicity of $u \in \us(\Omega)$ in terms of mean-value properties. For each $I \in \s, a \in \Omega, b \in \hh$ such that $\Omega$ includes the circle $\Gamma_{I,a,b}:=\{a+e^{I\vartheta}b: \vartheta\in\rr\}$, we will use the notation
\begin{equation}
l_I(u; a, b) := \frac{1}{2\pi} \int_0^{2\pi} u(a+e^{I\vartheta}b) d\vartheta.
\end{equation}

\begin{proposition}\label{weaklymean}
Let $\Omega$ be a domain in $\hh$ and let $u \in \us(\Omega)$. $u$ is weakly subharmonic (resp., weakly harmonic) if, and only if, the inequality 
\[u(a) \leq l_I(u; a, b)\]
(resp., the equality $u(a) = l_I(u; a, b)$) holds for all $I \in \s, a \in \Omega_I, b\in L_I\setminus\{0\}$ such that $\Gamma_{I,a,b}\subset\Omega_I$.
\end{proposition}

\begin{proof}
Fix any $I\in\s$. By~\cite[Theorem 2.4.1]{klimek}, $u_I$ is subharmonic in $\Omega_I$ if, and only if, $u(a) \leq l_I(u; a, b)$ for all $a \in \Omega_I, b\in L_I$ such that $\Gamma_{I,a,b}\subset\Omega_I\setminus\{0\}$. The corresponding equalities characterize harmonicity.
\end{proof}

\begin{proposition}\label{stronglymean}
Let $\Omega$ be a domain in $\hh$ and let $u \in \us(\Omega)$. $u$ is strongly subharmonic (resp., strongly harmonic) if, and only if, the inequality 
\[u(a) \leq l_I(u; a, b)\]
(resp., the equality $u(a) = l_I(u; a, b)$) holds for all $I \in \s, a \in \Omega, b \in \hh\setminus\{0\}$ such that $\Gamma_{I,a,b}\subset\Omega$.
\end{proposition}

\begin{proof}
For each $I\in\s$, let us apply~\cite[Theorem 2.9.1]{klimek} to establish whether $u$ is $I$-plurisubharmonic. This happens if, and only if, $u(a) \leq l_I(u; a, b)$ for all $a \in \Omega, b\in \hh\setminus\{0\}$ such that $\Gamma_{I,a,b}\subset\Omega$. The corresponding equalities characterize $I$-pluriharmonicity.
\end{proof}

As an application of the previous results, we can extend Theorem~\ref{compositionstrongregularc2} to all $u \in \us(\Omega)$.

\begin{theorem}\label{compositionstrongregular}
Let $u \in \us(\Omega)$. $u$ is strongly (sub)harmonic if, and only if, for every regular function $f : \Omega' \to \Omega$ the composition $u \circ f$ is weakly (sub)harmonic.
\end{theorem}

\begin{proof}
Let us suppose the composition $u \circ f$ with any regular function $f:\Omega' \to\Omega$ to be weakly subharmonic and let us prove that $u$ is strongly subharmonic. By Proposition~\ref{stronglymean}, it suffices to prove that, for any $I \in\s, a\in\Omega, b\in\hh\setminus\{0\}$ such that $\Gamma_{I,a,b}\subset\Omega$, the inequality
\[u(a) \leq l_I(u; a, b)\]
holds. If we set $f(q):=a+qb$, then $f(0)=a$ and $f$ maps the circle $\Gamma_{I,0,1}$ into the circle $\Gamma_{I,a,b}$. Thus, it suffices to prove that
\[u(f(0)) \leq l_I(u\circ f; 0, 1).\]
But this inequality is true by Proposition~\ref{weaklymean}, since $u \circ f$ is weakly subharmonic in a domain $\Omega'$ such that $\Gamma_{I,0,1}\subset\Omega'_I$. Analogous considerations can be made for the harmonic case.

Conversely, let $u \in \us(\Omega)$ be strongly (sub)harmonic, let $f : \Omega' \to \Omega$ be a regular function and let us prove that $u \circ f$ is weakly (sub)harmonic. For each $I\in\s$, $u$ is $I$-pluri(sub)harmonic and the restriction $f_I$ is a holomorphic map from $(\Omega'_I,I)$ to $(\Omega,I)$. As a consequence, $(u \circ f)_I=u\circ f_I$ is (sub)harmonic in $\Omega'_I$, as desired.
\end{proof}

A form of maximum modulus principle holds for weakly or strongly plurisubharmonic functions.

\begin{proposition}
Let $\Omega$ be a domain in $\hh$ and suppose $u \in \us(\Omega)$ to be weakly subharmonic. If $u$ has a local maximum point $p \in \Omega_I$ then $u_I$ is constant in the connected component of $\Omega_I$ that includes $p$. If, moreover, $u$ is strongly subharmonic, then $u$ is constant in $\Omega$.
\end{proposition}

\begin{proof}
In our hypotheses, $u_I$ is a subharmonic function with a local maximum point $p \in \Omega_I$. Thus, $u_I$ is constant in the connected component of $\Omega_I$ that includes $p$ by the maximum modulus principle for subharmonic functions~\cite[Theorem 2.4.2]{klimek}.

If, moreover, $u$ is strongly subharmonic then it is $I$-plurisubharmonic. Since we assumed $\Omega$ to be connected, $u$ is constant in $\Omega$ by the maximum modulus principle for plurisubharmonic functions~\cite[Corollary 2.9.9]{klimek}.
\end{proof}

Let us now consider maximality.

\begin{definition}
Let $S$ be a class of real-valued functions on an open set $D$ and let $\upsilon$ be an element of $S$. Suppose that, for any relatively compact subset $G$ of $D$ and for all $\nu\in S$ with $\nu \leq \upsilon$ in $\partial G$, the inequality $\nu \leq \upsilon$ holds throughout $G$. In this situation, we say that $\upsilon$ is \emph{maximal} in $S$ (or among the elements of $S$).
\end{definition}

The following characterization of weak harmonicity immediately follows from \cite[\S3.1]{klimek}.

\begin{remark}
Let $\Omega$ be a domain in $\hh$ and let $u\in\wsh(\Omega)$. $u$ is weakly harmonic if, and only if, for all $I \in \s$, the restriction $u_I$ is maximal among subharmonic functions on $\Omega_I$.
As a consequence, if $u$ is weakly harmonic then $u$ is maximal in $\wsh(\Omega)$.
\end{remark}

Let us now consider strongly subharmonic functions. Since they are plurisubharmonic with respect to all constant structures, we can make the following observation.

\begin{remark}
Let $\Omega$ be a domain in $\hh$ and let $u\in \ssh(\Omega)$. If $u$ is strongly harmonic then it is maximal in $\ssh(\Omega)$. Furthermore, if $u \in C^2(\Omega)$ then $u$ is maximal in $\ssh(\Omega)$ if and only if $\det H_{I,J}(u) \equiv 0$ for all $I,J \in \s$ with $I \perp J$.
\end{remark}

It is easy to exhibit a maximal element of $\ssh(\Omega)$ that is not strongly harmonic.

\begin{example}
The function $u(q) = \log|q|$ is a strongly subharmonic function on $\hh$. The explicit computations in Example \ref{log} show that $u$ maximal but not strongly harmonic. 
\end{example}


\section{Approximation}\label{sectionmapproximation}

An approximation result holds for strongly subharmonic functions. For all $\varepsilon >0$, let
\[\Omega_{\varepsilon}:=\left\{ 
\begin{array}{ll}
\{q \in \Omega : \mathrm{dist}(q, \partial \Omega)>\varepsilon\}&\mathrm{\ if\ }\Omega\neq\hh\\
\,\hh&\mathrm{\ if\ }\Omega=\hh
\end{array}
\right.\]
and let $u*\chi_\varepsilon$ denote the convolution of $u$ with the standard smoothing kernel $\chi_\varepsilon$ of $\hh\cong\rr^4$. For further details, see \cite[\S2.5]{klimek}.

\begin{proposition}\label{approximation}
Let $u \in \ssh(\Omega)$. If $\varepsilon>0$ is such that $\Omega_{\varepsilon}$ is not empty, then $u*\chi_\varepsilon \in C^{\infty}  \cap \ssh(\Omega_\varepsilon)$. Moreover, $u*\chi_\varepsilon$ monotonically decreases with decreasing $\varepsilon$ and 
\begin{equation}
\lim_{\varepsilon \to 0^+} u*\chi_\varepsilon (q) = u(q)
\end{equation}
for each $q \in \Omega$.
\end{proposition}

\begin{proof}
Fix any $I\in\s$, so that $u$ is $I$-plurisubharmonic. By~\cite[Theorem 2.9.2]{klimek}, $u*\chi_\varepsilon$ is both an element of $C^{\infty}$ and an $I$-plurisubharmonic function. For the same reason, our second statement also holds true.
\end{proof}

On the other hand, convolution with the standard smoothing kernel $\chi_\varepsilon$ does not preserve weak subharmonicity. This can be shown with an example.

\begin{example}
The function $u(q) = Re(q^2)$ is in $\wsh(\hh)$, but $u*\chi_\varepsilon$ does not belong to $\wsh(\hh)$. Indeed, we saw that for each orthonormal basis $1,I,J,IJ$ of $\hh$ we have $H_{I,J}(u)\equiv
\left( \begin{array}{cc}
0 & 0\\
0 & -1
\end{array}\right)$. Hence, $-u$ is strongly subharmonic and the same is true for $-u*\chi_\varepsilon$ by the previous proposition. In particular, $-u*\chi_\varepsilon \in \wsh(\hh)$ so that $u*\chi_\varepsilon$ can only be in $\wsh(\hh)$ if it is weakly \emph{harmonic}. This amounts to requiring that for each $I \in \s, a \in \Omega_I, b \in L_I\setminus\{0\}$ such that $\Gamma_{I,a,b} \subset \Omega$ the equality
\begin{equation*}
u*\chi_\varepsilon(a) = l_I(u*\chi_\varepsilon; a, b) = l_I(u; \cdot, b) * \chi_\varepsilon (a),
\end{equation*}
holds. But this happens if, and only if, $u(a-q) = l_I(u; a-q, b)$ for all $q$ in the support $\overline{B(0,\varepsilon)}$ of $\chi_\varepsilon$. This cannot be true, since (if $z_1,z_2$ denote the complex variables with respect to the orthonormal basis $1,I,J,IJ$) the function $-u$ is strictly subharmonic in $z_2$.
\end{example}


\section{Green's functions}

We now consider the analogs of Green's functions in the context of weakly and strongly subharmonic functions.

\begin{definition}\label{green}
Let $\Omega$ be a domain in $\hh$, let $q_0 \in \Omega$, and set
\begin{align*}
&\wsh_{q_0}(\Omega):=\left\{u \in \wsh(\Omega) : u<0, \limsup_{q \to q_0} \big|u(q)-\log|q-q_0|\big| < \infty\right\}\\
&\ssh_{q_0}(\Omega):=\wsh_{q_0}(\Omega) \cap \ssh(\Omega)\,.
\end{align*}
For all $q\in\Omega$, let us define
\begin{align*}
w(q) &:= \left\{
\begin{array}{ll}
-\infty& \mathrm{\ if\ }\wsh_{q_0}(\Omega) = \emptyset \\
\sup\{u(q): u\in \wsh_{q_0}(\Omega) \} &\mathrm{\ otherwise}
\end{array}
\right.\\
s(q) &:= \left\{
\begin{array}{ll}
-\infty& \mathrm{\ if\ }\ssh_{q_0}(\Omega) = \emptyset \\
\sup\{u(q): u\in \ssh_{q_0}(\Omega) \} &\mathrm{\ otherwise}
\end{array}
\right.
\end{align*}
 The \emph{Green function} of $\Omega$ with logarithmic pole at $q_0$, denoted $g^\Omega_{q_0}$, is the upper semicontinuous regularization $w^*$ of $w$. The \emph{strongly subharmonic Green function} of $\Omega$ with logarithmic pole at $q_0$, denoted $G^\Omega_{q_0}$, is the upper semicontinuous regularization $s^*$ of $s$.
\end{definition}

\begin{remark}
By construction, $G^\Omega_{q_0}(q) \leq g^\Omega_{q_0}(q)$ for all $q\in\Omega$. Moreover, any inclusion $\Omega'\subseteq\Omega$ implies $g^\Omega_{q_0}(q) \leq g^{\Omega'}_{q_0}(q)$ and $G^\Omega_{q_0}(q) \leq G^{\Omega'}_{q_0}(q)$.
\end{remark}

Let us construct a basic example. We will use the notations $\B := B(0,1)$, where 
\begin{equation*}
B(q_0,R) := \{q \in \hh : |q-q_0|<R\}
\end{equation*}
for all $q_0 \in \hh, R>0$, and $\B_I := \B \cap L_I$.

\begin{example}
We can easily prove that
\[G^\B_{0}(q)=g^\B_{0}(q)=\log|q|\]
for all $q\in\B$. Indeed, $q\mapsto\log|q|$ is clearly an element of $\ssh_{0}(\B)\subseteq\wsh_{0}(\B)$. Furthermore, for each $u\in\wsh_{0}(\B)$, the inequality $u(q)\leq\log|q|$ holds throughout $\B$. Indeed, for all $I\in\s$ it holds $u_I(z)\leq\log|z|$ for all $z\in\B_I$ because $z \mapsto \log|z|$ is the (complex) Green function of the disc $\B_I$.
\end{example}

Further examples can be derived by means of the next results.

\begin{lemma}\label{stonggreenequality}
Let $f$ be any affine transformation of $\hh$, let $\Omega$ be a domain in $\hh$ and fix $q_0 \in \Omega$. Then
\[G^{f(\Omega)}_{f(q_0)}(f(q)) = G^{\Omega}_{q_0}(q)\]
for all $q\in\Omega$.
\end{lemma}

\begin{proof}
By repeated applications of Remark \ref{affine}, we conclude that
\[\ssh_{q_0}(\Omega)=\{u\circ f : u\in \ssh_{f(q_0)}(f(\Omega))\}.\]
Thanks to this equality, the statement immediately follows from Definition~\ref{green}.
\end{proof}

\begin{lemma}\label{greeninequalities}
Let $\Omega$ be a symmetric slice domain in $\hh$, fix $q_0 \in \Omega$ and take a regular function $f:\Omega\to\hh$.
Then 
\[G^{f(\Omega)}_{f(q_0)}(f(q)) \leq g^{\Omega}_{q_0}(q)\]
for all $q\in\Omega$. If, moreover, $f$ is slice preserving then 
\[g^{f(\Omega)}_{f(q_0)}(f(q)) \leq g^{\Omega}_{q_0}(q)\]
for all $q\in\Omega$. If, additionally, $f$ admits a regular inverse $f^{-1}:f(\Omega)\to\Omega$, then the last inequality becomes an equality at all $q\in\Omega$.
\end{lemma}

\begin{proof}
By Theorem~\ref{compositionstrongregular},
\[\wsh_{q_0}(\Omega)\supseteq\{u\circ f : u\in \ssh_{f(q_0)}(f(\Omega))\}.\]
If $f$ is a slice preserving regular function then, by Theorem~\ref{compositionweakregular},
\[\wsh_{q_0}(\Omega)\supseteq\{u\circ f : u\in \wsh_{f(q_0)}(f(\Omega))\}.\]
The last inclusion is actually an equality if $f$ admits a regular inverse $f^{-1}:f(\Omega)\to\Omega$, which is necessarily slice preserving.
The three statements now follow from Definition~\ref{green}.
\end{proof}

In the last statement, we assumed $\Omega$ to be a symmetric slice domain for the sake of simplicity. The result could, however, be extended to all slice domains.

Lemmas~\ref{stonggreenequality} and~\ref{greeninequalities}, along with the preceding example, yield what follows.

\begin{example}
For each $x_0\in\rr$ and each $R>0$, the equalities
\[\log\frac{|q-x_0|}{R}=G^{B(x_0,R)}_{x_0}(q)=g^{B(x_0,R)}_{x_0}(q)\]
hold for all $q\in B(x_0,R)$. We point out that this function is strongly subharmonic and weakly harmonic in $B(x_0,R)$.
\end{example}

\begin{example}\label{ball}
For each $q_0\in\hh$ and each $R>0$ it holds
\[\log\frac{|q-q_0|}{R}=G^{B(q_0,R)}_{q_0}(q)\leq g^{B(q_0,R)}_{q_0}(q)\]
for all $q\in B=B(q_0,R)$. We point out that $G^{B}_{q_0}$, though strongly and weakly subharmonic, is not weakly \emph{harmonic} if $q_0 \not \in \rr$. Indeed, if we fix an orthonormal basis $1,I,J,IJ$ and write $q_0 = q_1+ q_2J$ with $q_1,q_2 \in L_I$, then by Lemma \ref{laplacian}
\[\left(\bar \partial_1 \partial_1 G^B_{q_0}\right)_{|_z} = \frac1{R^2} \left( \bar \partial_1 \partial_1 G^\B_{0}\right)_{|_{\frac{z-q_0}R}} = \frac{|q_2|^2}{2 (|z-q_1|^2 + |q_2|^2)^2}\]
for $z \in L_I$. This expression only vanishes when $q_2 = 0$, that is, when $q_0 \in L_I$.
\end{example}

We are now in a position to make the next remarks.

\begin{remark}
Let $\Omega$ be a bounded domain in $\hh$ and let $q_0 \in \Omega$. For all $r,R>0$ such that $B(q_0,r) \subseteq \Omega \subseteq B(q_0,R)$, we have that 
\[\log\frac{|q-q_0|}{R} \leq G^\Omega_{q_0}(q)\leq \log\frac{|q-q_0|}{r}.\]
As a consequence, $G^\Omega_{q_0}$ is not identically equal to $-\infty$, it belongs to $\ssh_{q_0}(\Omega)$ and it coincides with the supremum $s$ appearing in Definition \ref{green}.
\end{remark}

\begin{remark}
Let $\Omega$ be a bounded domain in $\hh$ and let $q_0 \in \Omega$. For all $R>0$ such that $\Omega \subseteq B(q_0,R)$, we have that 
\[\log\frac{|q-q_0|}{R} \leq g^\Omega_{q_0}(q)\,,\]
whence $g^\Omega_{q_0}$ is not identically equal to $-\infty$. If, moreover, $q_0=x_0\in\rr$ then for all $r>0$ such that $B(x_0,r) \subseteq \Omega$ we have that 
\[g^\Omega_{x_0}(q)\leq \log\frac{|q-x_0|}{r}.\]
In this case, $g^\Omega_{x_0}$ belongs to $\wsh_{x_0}(\Omega)$ and it coincides with the supremum $w$ appearing in Definition \ref{green}.
\end{remark}

When $\Omega$ admits a well-behaved exhaustion function, we can prove a few further properties.

\begin{theorem}
Let $\Omega$ be a bounded domain in $\hh$. Suppose there exists $\rho\in C^0(\Omega, (-\infty,0)) \cap \ssh(\Omega)$ such that $\{q \in \Omega: \rho(q) <c\} \subset\subset \Omega$ for all $c<0$. Then for all $q_0 \in \Omega$ and for all $p \in \partial \Omega$
\[\lim_{q \to p} G^\Omega_{q_0}(q) =0.\]
Moreover, $G^\Omega_{q_0}$ is continuous in $\Omega\setminus\{q_0\}$.
\end{theorem}

\begin{proof}
Let $B(q_0,r) \subseteq \Omega \subseteq B(q_0,R)$ and let $c>0$ be such that $c \rho < \log \frac r R$ in $\overline{B(q_0,r)}$. If we set
\begin{equation*}
u(q) =
\left\{
\begin{array}{ll}
\log\frac{|q-q_0|}{R} & q \in \overline{B(q_0,r)}\\
\max\left\{c \rho(q), \log\frac{|q-q_0|}{R}\right\} & q \in \Omega \setminus B(q_0,r)
\end{array}
\right.
\end{equation*}
then $u \in \ssh_{q_0}(\Omega)$. Thus, $u \leq G^\Omega_{q_0}\leq 0$, where $\lim_{q \to p} u(q) =0$ for all $p \in \partial \Omega$. This prove the first statement.

To prove the second statement, we only need to prove the lower semicontinuity of $G^\Omega_{q_0}$. Let us choose $\lambda \in (0,1)$ such that $\rho < -\lambda$ in $\overline{B(q_0,\lambda)}$. For any $\varepsilon \in (0,\lambda)$ such that $\log \frac \varepsilon R>(1-\varepsilon) \log \varepsilon^2$ (whence $\log \frac \varepsilon R > \varepsilon - \frac 1 \varepsilon$), let us set
\begin{equation*}
\alpha(q) = (1-\varepsilon) \log\left(\varepsilon |q-q_0|\right) - \varepsilon
\end{equation*}
on $\overline{B(q_0, \varepsilon)}$.

By Proposition~\ref{approximation}, for each sufficiently small $\delta>0$, the convolution $G^\Omega_{q_0} * \chi_\delta$ is an element of $C^\infty\cap\ssh(\Omega_\delta)$. If we choose $\eta \in (0,\varepsilon)$ so that $(1-\varepsilon) \log (\varepsilon \eta)>\log \frac \eta R$ then we may choose $\delta = \delta_\varepsilon>0$ so that:
\begin{itemize}
\item $(1-\varepsilon) \log (\varepsilon \eta)>G^\Omega_{q_0} * \chi_\delta$ in $\partial B(q_0,\eta)$,
\item $\Omega_\delta$ includes $\rho^{-1}([-\infty,-\varepsilon^3])$, and
\item $G^\Omega_{q_0} * \chi_\delta<0$ in $\rho^{-1}(-\varepsilon^3)$.
\end{itemize}
We may then set
\[\beta(q) :=  G^\Omega_{q_0} * \chi_{\delta_\varepsilon}(q) - \varepsilon\]
for all $q\in\Omega_{\delta_\varepsilon}$ and
\[\gamma(q) :=  \varepsilon^{-2} \rho(q)\]
for all $q\in\Omega$. By construction, $\alpha,\beta,\gamma$ can be patched together in a continuous strongly subharmonic function defined on $\Omega$, namely
\begin{equation*}
u_\varepsilon :=
\left\{
\begin{array}{ll}
\alpha & \mathrm{in\ } \overline{B(q_0,\eta)}\\
\max\left\{\alpha,\beta\right\} & \mathrm{in\ } \overline{B(q_0,\varepsilon)} \setminus B(q_0,\eta)\\
\beta & \mathrm{in\ } \rho^{-1}([-\infty,-\varepsilon]) \setminus B(q_0,\varepsilon)\\
\max\left\{\beta,\gamma\right\} & \mathrm{in\ } \rho^{-1}([-\varepsilon,-\varepsilon^3])\\
\gamma & \mathrm{in\ } \Omega \setminus \rho^{-1}([-\infty,-\varepsilon^3))
\end{array}
\right.
\end{equation*}

We remark that $\rho^{-1}([-\infty,-\varepsilon]) \setminus B(q_0,\varepsilon)$ increases as $\varepsilon \to 0$ and that
\[\bigcup_{\varepsilon \in (0,\lambda)}\left( \rho^{-1}([-\infty,-\varepsilon]) \setminus B(q_0,\varepsilon) \right)= \Omega \setminus \{q_0\}\,.\]
Thus, for all $q \in \Omega \setminus \{q_0\}$,
\[\lim_{\varepsilon \to 0} u_\varepsilon(q) = \lim_{\varepsilon \to 0} G^\Omega_{q_0} * \chi_{\delta_\varepsilon}(q) - \varepsilon = G^\Omega_{q_0}(q)\,.\]
Moreover, for each $\varepsilon \in (0, \lambda)$ it holds $\frac{u_\varepsilon}{1-\varepsilon} \in \ssh_{q_0}(\Omega)$, whence $\frac{u_\varepsilon}{1-\varepsilon} \leq G^\Omega_{q_0}$ in $\Omega$. Thus,
\begin{equation*}
G^\Omega_{q_0}(q) = \sup_{\varepsilon \in (0, \lambda)} \frac{u_\varepsilon(q)}{1-\varepsilon}\,,
\end{equation*}
whence the lower semicontinuity of $G^\Omega_{q_0}$ immediately follows.
\end{proof}

Similarly:

\begin{proposition}
Let $\Omega$ be a bounded domain in $\hh$. Suppose there exists $\rho\in C^0(\Omega, (-\infty,0)) \cap \wsh(\Omega)$ such that $\{q \in \Omega: \rho(q) <c\} \subset\subset \Omega$ for all $c<0$. Then for all $q_0 \in \Omega$ and for all $p \in \partial \Omega$
\[\lim_{q \to p} g^\Omega_{q_0}(q) =0.\]
\end{proposition}

\begin{proof}
Let $B(q_0,r) \subseteq \Omega \subseteq B(q_0,R)$ and let $c>0$ be such that $c \rho < \log \frac r R$ in $\overline{B(q_0,r)}$. If we set
\begin{equation*}
u(q) =
\left\{
\begin{array}{ll}
\log\frac{|q-q_0|}{R} & q \in \overline{B(q_0,r)}\\
\max\left\{c \rho(q), \log\frac{|q-q_0|}{R}\right\} & q \in \Omega \setminus B(q_0,r)
\end{array}
\right.
\end{equation*}
then $u \in \wsh_{q_0}(\Omega)$. Thus, $u \leq g^\Omega_{q_0}\leq 0$, where $\lim_{q \to p} u(q) =0$ for all $p \in \partial \Omega$.
\end{proof}

For a special class of domains $\Omega$ and points $q_0$, the Green function $g^\Omega_{q_0}$ can be easily determined, as follows.

\begin{theorem}\label{greenonsymmetricslice}
Let $\Omega$ be a bounded symmetric slice domain and let $x_0 \in \Omega \cap \rr$. Consider the slice $\Omega_i = \Omega \cap \cc$ of the domain and the (complex) Green function of $\Omega_i$ with logarithmic pole at $x_0$, which we will denote as $\gamma^{\Omega_i}_{x_0}$. If we set
\[u(x+Iy) := \gamma^{\Omega_i}_{x_0}(x+iy)\]
for all $x,y\in\rr$ and $I \in \s$ such that $x+Iy\in\Omega$, then:
\begin{itemize}
\item $u$ is a well-defined function on $\Omega$;
\item $u$ is $\jj$-plurisubharmonic in $\Omega \setminus \rr$ and it belongs to $\wsh_{x_0}(\Omega)$;
\item $u$ coincides with $g^\Omega_{x_0}$.
\end{itemize}
\end{theorem}

\begin{proof}
The slice $\Omega_i$ of the domain is a bounded domain in $\cc$. By~\cite[Proposition 6.1.1]{klimek}, the function $\gamma^{\Omega_i}_{x_0}$ is a negative plurisubharmonic function on $\Omega_i$ with a logarithmic pole at $x_0$. Moreover, since $\Omega_i$ is symmetric with respect to the real axis, it holds $\gamma^{\Omega_i}_{x_0}(x+iy)=\gamma^{\Omega_i}_{x_0}(x-iy)$ for all $x+iy\in\Omega_i$.
It follows at once that $u$ is well-defined, that it belongs to $\wsh_{x_0}(\Omega)$, and that it is $\jj$-plurisubharmonic in $\Omega\setminus\rr$.

Moreover, let us fix any other $v\in \wsh_{0}(\B)$: we can prove that $v(q) \leq u(q)$ for all $q\in\Omega$, as follows. For each $I \in \s$, the inequality $v_I\leq u_I$ holds in $\Omega_I$ because (up to identifying $L_I$ with $\cc$) the function $u_I$ is the (complex) Green function of $\Omega_I$ with logarithmic pole at $x_0$ and $v_I$ is a negative subharmonic function on $\Omega_I$ with a logarithmic pole at $x_0$. As a consequence, $u$ coincides with $g^\Omega_{x_0}$.
\end{proof}


\subsection{A significant example}

An interesting example to consider is that of the unit ball $\B$ with a pole other than $0$. It is natural to address it by means of the \emph{classical M\"obius transformations} of $\B$, namely the conformal transformations $v^{-1}M_{q_0}u$, where $u,v$ are constants in $\partial\B$, $q_0\in\B$ and $M_{q_0}$ is the transformation of $\B$ defined as
\[M_{q_0}(q) := (1-q\bar q_0)^{-1}(q-q_0)\,.\]
The transformation $M_{q_0}$ has inverse $M_{q_0}^{-1}=M_{-q_0}$. It is regular if, and only if, $q_0=x_0\in\rr$; in this case, it is also slice preserving. For more details, see~\cite{poincare,moebius}.

Our first observation can be derived from either Lemma~\ref{greeninequalities} or Theorem~\ref{greenonsymmetricslice}.

\begin{example}
For each $x_0\in\B\cap\rr$, we have
\[G^{\B}_{x_0}(q)\leq g^{\B}_{x_0}(q)=\log\frac{|q-x_0|}{|1-qx_0|}\]
for all $q\in \B$.
\end{example}

The same techniques do not work when the logarithmic pole $q_0$ is not real, as a consequence of the fact that $M_{q_0}$ is not a regular function. Nevertheless, we can make the following observation.

\begin{example}
Let us fix $q_0\in\B\setminus\rr$. We will prove that 
\[\log\frac{\left|q-q_0\right|}{\left|1-q\bar q_0\right|}\leq g^\B_{q_0}(q)\]
by showing that $u(q) = \log\frac{\left|q-q_0\right|}{\left|1-q\bar q_0\right|}$ is weakly subharmonic. We will also prove that: (a) the restriction $u_I$ is harmonic if, and only if, $q_0\in\B_I$; (b) the function $u$ is not strongly subharmonic.
\begin{itemize}
\item We first observe that 
\[u(q)=\log|q-q_0|-\log|q-\tilde q_0|-\log|\bar q_0|\,,\]
with $\tilde q_0:=\bar q_0^{-1} = q_0|q_0|^{-2}$. With respect to any orthonormal basis $1,I,J,IJ$ and to the associated coordinates $z_1,z_2,\bar z_1,\bar z_2$, if we split $q_0,\tilde q_0$ as $q_0 = q_1 + q_2J,\tilde q_0=\tilde q_1 + \tilde q_2J$ with $q_1,q_2,\tilde q_1,\tilde q_2 \in \B_I$ then
\begin{align*}
H_{I,J}(u)_{|_q}&=
\frac{1}{2|q-q_0|^4} 
\left( \begin{array}{cc}
|z_2-q_2|^2 & -(z_1-q_1)(\bar z_2-\bar q_2)\\
-(z_2-q_2)(\bar z_1-\bar q_1) & |z_1-q_1|^2
\end{array}\right)+\\
&-\frac{1}{2\left|q-\tilde q_0\right|^4} 
\left( \begin{array}{cc}
|z_2-\tilde q_2|^2 & -(z_1-q_1)(\bar z_2-\overline{\tilde q_2})\\
-(z_2-\tilde q_2)(\bar z_1-\overline{\tilde q_1}) & |z_1-\tilde q_1|^2
\end{array}\right)\,.
\end{align*}
\item For all $z\in\B_I$ (that is, $z_1=z,z_2=0$) it holds
\begin{align*}
\left(\bar \partial_1 \partial_1 u \right)_{|_z} &= \frac{|q_2|^2}{2 (|z-q_1|^2 + |q_2|^2)^2} - \frac{|q_2|^2 |q_0|^4}{2 \left(\big|z\,|q_0|^2-q_1\big|^2 + |q_2|^2\right)^2}\,.
\end{align*}
If $q_0\in\B_I$ then $q_2=0$ and $\left(\bar \partial_1 \partial_1 u \right)$ vanishes identically in $\B_I$. Otherwise, $\left(\bar \partial_1 \partial_1 u \right)_{|_z}>0$ for all $z \in \B_I$ because
\begin{align*}
&\big|z\,|q_0|^2-q_1\big|^2 + |q_2|^2-|q_0|^2(|z-q_1|^2 + |q_2|^2)\\
&=|z|^2|q_0|^4+|q_0|^2-|q_0|^2|z|^2-|q_0|^4\\
&=|q_0|^2(1-|q_0|^2)(1-|z|^2)>0\,.
\end{align*}
\item In general, $H_{I,J}(u)$ is not positive semidefinite. To see this, let us choose a basis $1,I,J,IJ$ so that $q_1\neq0\neq q_2$ and let us choose $z_1=0,z_2=q_2$. We get
\begin{align*}
H_{I,J}(u)_{|_{q_2J}}&=
\frac{1}{2|q_1|^4} 
\left( \begin{array}{cc}
0 & 0\\
0& |q_1|^2
\end{array}\right)+\\
&-\frac{1}{2\left(|\tilde q_1|^2+|q_2-\tilde q_2|^2\right)^2} 
\left( \begin{array}{cc}
|q_2-\tilde q_2|^2 & q_1(\bar q_2-\overline{\tilde q_2})\\
(q_2-\tilde q_2)\overline{\tilde q_1} & |\tilde q_1|^2
\end{array}\right)\,,
\end{align*}
where $q_2-\tilde q_2\neq0$ by construction.
\end{itemize}
\end{example}

We would now like to consider a different approach, through regular transformations of $\B$. Indeed, the work~\cite{moebius} proved the following facts.
\begin{itemize}
\item The only regular bijections $\B\to\B$ are the so-called \emph{regular M\"obius transformations} of $\B$, namely the transformations $\M_{q_0}*u=\M_{q_0}u$ with $u\in\partial\B, q_0 \in\B$ and
\[\M_{q_0}(q) := (1-q\bar q_0)^{-*}*(q-q_0)\,.\]
Here, the symbol $*$ denotes the multiplicative operation among regular functions and $f^{-*}$ is the inverse of $f$ with respect to this multiplicative operation.
\item For all $q\in\B$, it holds
\[\M_{q_0}(q) = M_{q_0}(T_{q_0}(q))\,,\]
where $T_{q_0}: \B \to \B$ is defined as $T_{q_0}(q) = (1-qq_0)^{-1}q (1-qq_0)$ and has inverse $T_{q_0}^{-1}(q)=T_{\bar q_0}(q)$.
\item $\M_{q_0}$ is slice preserving if, and only if, $q_0=x_0\in\rr$ (in which case, $T_{q_0}={id}_\B$ and $\M_{x_0}=M_{x_0}$).
\end{itemize}
If we fix $q_0\in\B\setminus\rr$ then, by Lemma~\ref{greeninequalities},
\begin{equation}\label{greenball1}
\log \left|\M_{q_0}(q)\right|\leq g^\B_{q_0}(q)\,.
\end{equation}
for all $q\in\B$.
The work~\cite{schwarzpick} proved the quaternionic Schwarz-Pick Lemma and, in particular, the inequality
\begin{align*}
\log|f(q)| \leq \log \left|\M_{q_0}(q)\right|
\end{align*}
valid for all regular $f:\B \to \B$ with $f(q_0)=0$. It is therefore natural to ask ourselves whether an equality may hold in \eqref{greenball1}. However, this is not the case: as a consequence of the next result, inequality~\eqref{greenball1} is strict at all $q$ not belonging to the same slice $\B_I$ as $q_0$. As a byproduct, we conclude that the set
\[\{\log|f| : f : \B \to \hh\mathrm{\ regular,\ } f(q_0)=0\}\]
is not a dense subset of $\wsh_{q_0}(\B)$.

\begin{theorem}
If $q_0\in\B_I$, then
\begin{equation}\label{greenball2}
|M_{q_0}(T_{q_0}(q))|=|\M_{q_0}(q)|\leq|M_{q_0}(q)|\,,
\end{equation}
for all $q\in\B$. Equality holds if, and only if, $q\in\B_I$.
\end{theorem}

\begin{proof}
Inequality \eqref{greenball2} is equivalent to 
\[|T_{q_0}(q)-q_0|\,|1-q\bar q_0|\leq|q-q_0|\,|1-T_{q_0}(q)\bar q_0|\,.\]
Since $|T_{q_0}(q)|=|q|$, the last inequality is equivalent to
\begin{align*}
0\leq&\left(|q|^2-2\langle q,q_0 \rangle+|q_0|^2\right)\left(1-2\langle T_{q_0}(q),q_0 \rangle+|q|^2|q_0|^2\right)+\\
&-\left(|q|^2-2\langle T_{q_0}(q),q_0 \rangle+|q_0|^2\right)\left(1-2\langle q,q_0 \rangle+|q|^2|q_0|^2\right)\\
=&\,2\left(-|q|^2-|q_0|^2+1+|q|^2|q_0|^2\right)\langle T_{q_0}(q),q_0 \rangle+\\
&\,2\left(-1-|q|^2|q_0|^2+|q|^2+|q_0|^2\right)\langle q,q_0 \rangle\\
=&\,2(1-|q_0|^2)(1-|q|^2))\langle T_{q_0}(q)-q,q_0 \rangle\,.
\end{align*}
Thus, inequality \eqref{greenball2} holds for $q\in\B$ if, and only if, $0\leq\langle T_{q_0}(q)-q,q_0 \rangle$. This is equivalent to the non-negativity of the real part of $(T_{q_0}(q)-q)\bar q_0=((1-qq_0)^{-1}q(1-qq_0)-q)\bar q_0$ or, equivalently, of the real part of
\begin{align*}
&\left(\overline{(1-qq_0)}q(1-qq_0)-|1-qq_0|^2q\right)\bar q_0\\
&=\overline{(1-qq_0)}(q-q^2q_0-q+qq_0q)\bar q_0\\
&=(1-\bar q_0\bar q)(qq_0q\bar q_0-q^2|q_0|^2)\\
&=qq_0q\bar q_0-q^2|q_0|^2+|q|^2|q_0|^2(\bar q_0q-q\bar q_0)\\
&=(|q_0|^2-\bar q_0^2)(|z_2|^2-z_1z_2J)+(\bar q_0-q_0)z_2J\,,
\end{align*}
where the last equality can be obtained by direct computation after splitting $q$ as $q=z_1+z_2J$, with $z_1,z_2\in\B_I$ and $J\perp I$. If $q_0=x_0+Iy_0$ then the real part of the last expression equals $(x_0^2+y_0^2-x_0^2+y_0^2)|z_2|^2=2y_0^2|z_2|^2$, which is clearly non-negative. Moreover, it vanishes if, and only if, $z_2=0$, i.e., $q\in\B_I$.
\end{proof}

An example wherein inequality \eqref{greenball2} holds and is strict had been constructed in~\cite{volumeindam}. That construction was used to prove that regular M\"obius transformations are not isometries for the Poincar\'e distance of $\B$, defined as
\[\delta_\B(q,q_0):=\frac12\log\left(\frac{1+|M_{q_0}(q)|}{1-|M_{q_0}(q)|}\right)\]
for all $q,q_0\in\B$. Our new inequality~\eqref{greenball2} is equivalent to $\delta_\B(T_{q_0}(q),q_0)\leq\delta_\B(q,q_0)$. In other words, we have proven the following property of the transformation $T_{q_0}$ of $\B$: while all points $q\in\B_I$ are fixed, all points $q\in\B\setminus\B_I$ are attracted to $q_0$ with respect to the Poincar\'e distance.



\end{document}